\documentclass[11pt,a4paper]{article}
\usepackage[utf8]{inputenc}
\usepackage[english]{babel}
\usepackage{amsmath,amssymb,amsthm}
\usepackage{geometry}
\numberwithin{equation}{section}

\newtheorem{theorem}{Theorem}[section]
\newtheorem{lemma}[theorem]{Lemma}
\newtheorem{remark}[theorem]{Remark}

\title{\textbf{Giaquinta's example revisited: the Liouville property for a class of nonuniformly elliptic equations in low dimensions}}
\author{Michael Bildhauer \quad Martin Fuchs}
\date{}

\begin{document}

\maketitle

\begin{abstract}
We consider entire solutions $u$: $\mathbb{R}^n \rightarrow \mathbb{R}$, $n \ge 2$, of the equation $\operatorname{div}(\nabla f(\nabla u)) = 0$ 
with convex density $f$: $\mathbb{R}^n \rightarrow \mathbb{R}$ given for example by $f(\nabla u) = \sum_{i=1}^n |\partial_i u|^{p_i}$ with 
exponents $p_i \ge 2$ including Giaquinta's example for the choice $p_1 = \dots = p_{n-1} = 2$, $p_n = 4$. If $n \le 2 + \min\{p_i\}$ holds, then we have the Liouville property: the boundedness of $u$ implies its constancy.


\end{abstract}

\section{Introduction}

In 1987 Giaquinta \cite{Gia1987} showed that the function
\begin{equation*}
u(x) := \sqrt{\frac{n-4}{24}} \frac{x_n^2}{\left(\sum_{i=1}^{n-1} x_i^2\right)^{1/2}}, \quad x \in B := \{y \in \mathbb{R}^n : |y| < 1\},
\end{equation*}
for dimensions $n \ge 6$ has the property that
\begin{equation}
I[u] := \int_B \left( \sum_{i=1}^{n-1} (\partial_i u)^2 + \frac{1}{2}(\partial_n u)^4 \right) dx < \infty
\end{equation}
and in addition is a weak solution of the Euler-Lagrange equation of the variational integral $I$ minimizing with respect to its boundary data. 
Obviously, $u$ is an unbounded function. In contrast, Cupini, Marcellini and Mascolo proved in Theorem 1.1 of \cite{CMM2009}, 
that in the particular case of the functional $I$ defined in (1.1) local minimizers are locally bounded, if the case $n < 5$ is considered. Let us mention that in the paper actually very general functionals even with variable growth rates are discussed exhibiting conditions on the exponents and the dimension under which $u \in L_{\operatorname{loc}}^\infty$ holds for locally minimizing functions $u$.

As outlined for instance by Frehse \cite{Fre1979, Fre1977}, and discussed later by Giaquinta and Nečas \cite{GN1980} and Kawohl \cite{Kaw1980}, 
it is a generally accepted principle that a regularity theory holds for those kinds of equations for which a Liouville theorem is true. 
Motivated by Giaquinta's example the goal of our short note is to show that in the case of the 
functional (1.1) Liouville's theorem holds for $n \le 4$. More generally, suppose that
\begin{equation}
f(P) := \sum_{i=1}^n f_i(P_i), \quad P = (P_1, \dots, P_n) \in \mathbb{R}^n
\end{equation}
with functions $f_i \in C^2(\mathbb{R})$ satisfying
\begin{equation}
f_i'(0) = 0, \quad \lambda_i |t|^{p_i-2} \le f_i''(t) \le \Lambda_i |t|^{p_i-2} \quad t \in \mathbb{R}
\end{equation}
with constants $0 < \lambda_i \le \Lambda_i$ and with exponents
\begin{equation}
p_i \ge 2
\end{equation}
for any index $i \in \{1, \dots, n\}$. Then we have

\begin{theorem}
Let (1.2)--(1.4) hold and assume that $u \in C^2(\mathbb{R}^n)$ is a solution of
\begin{equation}
\operatorname{div}[\nabla f(\nabla u)] = \sum_{i=1}^n \partial_i (f_i'(\partial_i u)) = 0 \quad \text{on } \mathbb{R}^n.
\end{equation}
Suppose that we have for the exponents $p_i$
\begin{equation}
n \le 2 + \min \{p_i : i = 1, \dots, n\}.
\end{equation}
Then the Liouville property holds for equation (1.5), which means that bounded solutions must be constant.
\end{theorem}

\begin{remark}
In the case of Giaquinta's example we have $p_1 = \dots = p_{n-1} = 2$, $p_n = 4$. For (1.6) we need 
$n \le 4$, which is in correspondence to the regularity result of \cite{CMM2009}. 
We do not know, if for the particular density $f$ from (1.1) the Liouville property is true in dimension 5.
\end{remark}

\begin{remark}
Let the assumptions of Theorem 1.1 hold with the exception of (1.6). We will show that the hypothesis
\begin{equation}
\limsup_{|x| \rightarrow \infty} \frac{|u(x)|}{|x|^\beta} < \infty
\end{equation}
for the choice ($p_{\min} := \min \{p_i : i = 1, \dots, n\}$)
\begin{equation}
\beta := \frac{1}{p_{\min}} (p_{\min} + 2 - n)
\end{equation}
implies that the entire solution $u$ is an affine function. Now, if (1.6) is valid, we have $\beta \ge 0$, and if in addition we assume the boundedness of $u$, then (1.7) holds yielding $\nabla^2 u = 0$ which clearly implies the constancy of $u$.

In the case $n=2$ we get $\beta = 1$, which means: 
if $u \in C^2(\mathbb{R}^2)$ is an entire solution of equation (1.5) under the conditions (1.2)--(1.4), for which
\begin{equation*}
\limsup_{|x| \rightarrow \infty} |u(x)| / |x| < \infty,
\end{equation*}
then $u$ must be affine.
\end{remark}

\begin{remark}
Now we want to replace (1.3) by the requirement that
\[
\lambda_i (1+t^2)^{\frac{p_i-2}{2}} \le f_i''(t) \le \Lambda_i (1+t^2)^{\frac{p_i-2}{2}} \quad t \in \mathbb{R}
\]
holds for the functions $f_i$ (nondegenerate case). The result of Theorem 1.1 does not extend to this different setting: (1.6) has to be replaced by the stronger condition $n \le 4$, and in Remark 1.3 one has to work with the value $\beta = \frac{1}{2}(4-n)$ instead of (1.8), which means that the exponents $p_i \ge 2$, $i = 1, \dots, n$, do not explicitly occur. The interested reader is referred to the forthcoming paper \cite{BF_forth}.
\end{remark}

\section{Proof of Theorem 1.1}

We start with an inequality of Caccioppoli type, in which $\Phi : \mathbb{R} \rightarrow \mathbb{R}$ denotes a function of class $C^1$ with the properties
\begin{equation}
\begin{cases}
\Phi(t) = t^{2m} \text{ on } [-1,1] \quad \Phi(t) = 2 \text{ if } |t| \ge 2 \text{ and} \\
t \Phi'(t) \ge 0 \text{ for any } t \in \mathbb{R},
\end{cases}
\end{equation}
where $m \in \mathbb{N}$ will be fixed later.

\begin{lemma}
Let $F$: $\mathbb{R}^n \rightarrow \mathbb{R}$ be a convex function of class $C^2$ and consider an entire solution $v \in C^2(\mathbb{R}^n)$ of the equation $\operatorname{div}(\nabla F(\nabla v)) = 0$. Then we have for $i = 1, \dots, n$ and $\eta \in C_0^1(\mathbb{R}^n)$ such that $0 \le \eta \le 1$
\begin{eqnarray}
\lefteqn{\int_{\operatorname{spt}(\eta)} D^2 F(\nabla v)(\nabla \partial_i v, \nabla \partial_i v) \, \Phi(\partial_i v) \eta^2 \, dx} \nonumber \\[2mm]
&\leq & c \, \bigg(  \int_{\operatorname{spt}(\nabla \eta)} D^2 F(\nabla v)(\nabla \partial_i v, \nabla \partial_i v) \, \Phi(\partial_i v) \eta^2 \, dx \bigg)^{1/2} 
\cdot \nonumber \\[2mm]
&& \bigg( \int_{\operatorname{spt}(\nabla \eta)} D^2 F(\nabla v)(\nabla \eta, \nabla \eta) |\partial_i v|^2 \Phi(\partial_i v) \, dx \bigg)^{1/2}, \\[4mm]
&& \int_{\operatorname{spt}(\eta)} D^2 F(\nabla v)(\nabla \partial_i v, \nabla \partial_i v) \, \Phi(\partial_i v) \eta^2 \, dx \le \label{eq:lemma21_final} \\[2mm]
&&c \,\int_{\operatorname{spt}(\nabla \eta)} D^2 F(\nabla v)(\nabla \eta, \nabla \eta) |\partial_i v|^2 \Phi(\partial_i v) \, dx, \nonumber
\end{eqnarray}
$c$ denoting a constant independent of $v$ and $\eta$.
\end{lemma}

\begin{proof}[Proof of Lemma 2.1]
Clearly (2.3) follows from (2.2) by Young's inequality. For (2.2) we observe
\[
\int_{\operatorname{spt}(\eta)} D^2 f( \nabla v)\big(\nabla \partial_i v,\nabla[ \eta^2 \partial_i v \, \Phi(\partial_i v)]\big) \, dx = 0
\]
and therefore
\begin{eqnarray*}
\lefteqn{\int_{\operatorname{spt}(\eta)} D^2 F(\nabla v)(\nabla \partial_i v, \nabla \partial_i v) \eta^2 \Phi(\partial_i v) \, dx} \\
&&+ \int_{\operatorname{spt}(\eta)} D^2 F(\nabla v)(\nabla \partial_i v, \nabla \partial_i v) \eta^2 \Phi'(\partial_i v) \partial_i v \, dx \\
&= &-2 \int_{\operatorname{spt}(\nabla \eta)} D^2 F(\nabla v)(\nabla \partial_i v, \nabla \eta) \eta \partial_i v \, \Phi(\partial_i v) \, dx.
\end{eqnarray*}
By (2.1) we have that $\partial_i v \, \Phi'(\partial_i v) \ge 0$, hence
\begin{align*}
&\int_{\operatorname{spt}(\eta)} D^2 F(\nabla v)(\nabla \partial_i v, \nabla \partial_i v) \eta^2 \Phi(\partial_i v) \, dx \le \\
&-2 \int_{\operatorname{spt}(\nabla \eta)} D^2 F(\nabla v)(\nabla \partial_i v, \nabla \eta) \eta \partial_i v \, \Phi(\partial_i v) \, dx.
\end{align*}
Using the Cauchy-Schwarz inequality (in terms of the bilinear form $D^2 F(\nabla v)$) on the right-hand side of the last inequality and applying Hölder's inequality (observing that $\Phi \ge 0$ during the calculations), we get (2.2).
\end{proof}

In accordance with Remark 1.3 we assume that we are in the situation of Theorem 1.1, but with (1.6) being replaced by (1.7), 
where $\beta$ is taken from (1.8). After renumbering the variables $x_1, \dots, x_n$ we may assume that (recall (1.4))
\begin{equation}
2 \le p_1 \le p_2 \le \dots \le p_n
\end{equation}
holds. In (2.1) we choose $m \in \mathbb{N}$ such that
\begin{equation}
p_n \le (m+1) p_1.
\end{equation}
Consider a ball $B_{2R} := \{x \in \mathbb{R}^n : |x| < 2R\}$, choose $\eta \in C_0^1(B_{2R})$ such that 
$0 \le \eta \le 1$, $\eta = 1$ on $B_R$, $|\nabla \eta| \le c/R$, and apply (2.3) with the choices $u$, $f$ and $n$ replacing $v$, $F$ and $i$. 
Noting that (recall (1.2))
\begin{equation*}
D^2 f(P)(Q,Q) = \sum_{k=1}^n f_k''(P_k) Q_k^2 \quad P, Q \in \mathbb{R}^n
\end{equation*}
we deduce with the help of (1.3)
\begin{equation}
\text{r.h.s. of (2.3)} \le \label{eq:rhs_23_bound} 
c \int_{B_{2R}-B_R} \sum_{k=1}^n |\partial_k u|^{p_k-2} |\partial_k \eta|^2 \Phi(\partial_n u) |\partial_n u|^2 \, dx.
\end{equation}

Next we claim that for $k = 1, \dots, n$ it holds
\begin{equation}
a_k := |\partial_k u|^{p_k-2} |\partial_n u|^2 \Phi(\partial_n u) \le \label{eq:ak_bound} 
c \left[ |\partial_k u|^{p_k} + |\partial_n u|^{p_n} \right] 
\end{equation}
which for $k = n$ is immediate on account of $0 \le \Phi \le 2$ (compare (2.1)).\\ 

So let $k < n$ and observe that on the set $[|\partial_n u| \le 1]$ we have (again by (2.1))
\begin{equation*}
a_k = |\partial_k u|^{p_k-2} |\partial_n u|^{2+2m}.
\end{equation*}

If $p_k = 2$, then (2.5) shows $p_n \le 2(m+1)$, hence $a_k \le |\partial_n u|^{p_n}$ giving (2.7).\\

In the case $p_k > 2$ Young's inequality yields (noting that $(m+1)p_k \ge p_n$)
\begin{align*}
a_k &\le c \left( |\partial_k u|^{p_k} + |\partial_n u|^{(m+1)p_k} \right) \le\ c \left( |\partial_k u|^{p_k} + |\partial_n u|^{p_n} \right)
\end{align*}
thus we have (2.7) on the set $[|\partial_n u| \le 1]$.\\ 

On $[|\partial_n u| > 1]$ we have $\Phi(\partial_n u) \le 2$, so that
\begin{align*}
a_k &\le 2 |\partial_k u|^{p_k-2} |\partial_n u|^2
\end{align*}
and in conclusion $a_k \le 2 |\partial_n u|^2 \le 2 |\partial_n u|^{p_n}$, if the case $p_k = 2$ is considered.\\

Otherwise we again use Young's inequality to get
\begin{align*}
a_k &\le c \left( |\partial_k u|^{p_k} + |\partial_n u|^{p_k} \right) \ \le c \left( |\partial_k u|^{p_k} + |\partial_n u|^{p_n} \right)
\end{align*}
where we have used (2.4). This proves (2.7).\\

We apply (2.7) to inequality (2.6) with the result (using the properties of $\eta$ specified after (2.5))
\begin{equation}
\int_{B_R} D^2 f(\nabla u)(\nabla \partial_n u, \nabla \partial_n u) \Phi(\partial_n u)  dx \le c R^{-2} \int_{B_{2R}-B_R} \sum_{k=1}^n |\partial_k u|^{p_k}  dx.
\end{equation}
Let $\eta \in C_0^1(T_R)$, $T_R := B_{\frac{5}{2}R} - B_{\frac{R}{2}}$, such that $0 \le \eta \le 1$, 
$\eta = 1$ on $B_{2R}-B_R$ and $|\nabla \eta| \le c/R$. Clearly for $l \in \mathbb{N}$, it holds on account of (1.3) and (1.5)
\begin{eqnarray*}
\int_{T_R} \eta^{2l} \sum_{i=1}^n |\partial_i u|^{p_i}  dx 
& \le &c \int_{T_R} \eta^{2l} \nabla f(\nabla u) \cdot \nabla u    dx\\[2mm]  
& =& -2l  c \int_{T_R} \eta^{2l-1} \nabla \eta \cdot \nabla f(\nabla u) u dx\\[2mm]
&\le& \ c(l) \int_{T_R} \eta^{2l-1} 
\sum_{i=1}^n |\partial_i u|^{p_i-1} |\partial_i \eta|  |u|  dx .
\end{eqnarray*}
By Young's inequality we have
\begin{equation*}
\eta^{2l-1} |\partial_i u|^{p_i-1} |u| |\partial_i \eta| \le 
\varepsilon  \eta^{(2l-1)\frac{p_i}{p_i-1}} |\partial_i u|^{p_i} + c(\varepsilon) |\nabla \eta|^{p_i} |u|^{p_i}
\end{equation*}
and in addition $(2l-1)\frac{p_i}{p_i-1} \ge 2l$, if $l$ is large enough. Thus we end up with
\begin{equation}
\int_{T_R}\eta^{2l} \ \sum_{i=1}^n |\partial_i u|^{p_i}  dx \le c \sum_{i=1}^n R^{-p_i} 
\int_{\operatorname{spt}(\nabla \eta)} |u|^{p_i}  dx.
\end{equation}
On the right-hand side of (2.9) we use (1.7), which gives
\begin{equation*}
\sum_{i=1}^n R^{-p_i} \int_{\operatorname{spt}(\nabla \eta)} |u|^{p_i}  dx \le c \sum_{i=1}^n R^{n-p_i} R^{p_i \beta}
\end{equation*}
with $\beta$ from (1.8). Inserting this into (2.9) and returning to (2.8) (recall the choice of $\eta$), it follows
\begin{equation*}
\int_{B_R} D^2 f(\nabla u)(\nabla \partial_n u, \nabla \partial_n u) \Phi(\partial_n u)  dx \le c \sum_{i=1}^n R^{n-2-p_i+p_i \beta}.
\end{equation*}
Observing that $\beta \le 1$, we have
\begin{equation*}
n-2+p_i(\beta-1) \le n-2+p_{\min}(\beta-1) = 0
\end{equation*}
and the foregoing estimate implies (letting $R \rightarrow \infty$)
\begin{equation}
\int_{\mathbb{R}^n} D^2 f(\nabla u)(\nabla \partial_n u, \nabla \partial_n u) \Phi(\partial_n u)  dx < \infty.
\end{equation}
Next we return to (2.2) replacing $F, v, i$ by $f, u, n$ and choosing $\eta \in C_0^1(B_{2R})$, $\eta = 1$ on $B_R$, $|\nabla \eta| \le c/R$ and $0 \le \eta \le 1$.
The foregoing calculations show the boundedness of the second integral on the right-hand side of (2.2), 
whereas the first one vanishes as $R \rightarrow \infty$ on account of (2.10). This gives
\begin{equation}
D^2 f( \nabla u)(\nabla \partial_n u, \nabla \partial_n u) \Phi(\partial_n u) = 0 \quad \text{on } \mathbb{R}^n.
\end{equation}
On the set $[\partial_n u = 0]$ we have $\nabla \partial_n u = 0$ almost everywhere (see, e.g., \cite{GT2001}, p. 152, Lemma 7.7), 
whereas for $x \in [\partial_n u \neq 0]$ it holds $\Phi(\partial_n u(x)) \neq 0$, thus
\begin{equation}
f_i''(\partial_i u) (\partial_i \partial_n u)^2 = 0 \quad \text{on } [\partial_n u \neq 0]
\end{equation}
on account of (1.3) and (2.11). (2.12) gives $\partial_n \partial_n u = 0$. For $i \le n-1$ we see (again for $x \in [\partial_n u \neq 0]$) that
\begin{equation}
\partial_i  u(x) = 0 \quad \text{or} \quad \partial_i \partial_n u(x) = 0\end{equation}
and $\partial_n \partial_i u = 0$ is true almost everywhere on $[\partial_i u = 0]$. 
This shows $\partial_i \partial_n u = 0$ almost everywhere on $[\partial_n u \neq 0]$ now for $i = 1, \dots, n$, and therefore 
$\nabla \partial_n u = 0$ almost everywhere on $\mathbb{R}^n$.\\

The continuity of $\nabla \partial_n u$ shows $\nabla \partial_n u = 0$, and we find a constant $c_n$ such that $\partial_n u = c_n$. 
We have for $x = (x_1, \dots, x_n)$
\begin{equation*}
u(x) = u(x_1, \dots, x_{n-1}, 0) + c_n x_n
\end{equation*}
and the function $v(x_1, \dots, x_{n-1}) := u(x_1, \dots, x_{n-1}, 0)$ is an entire solution of (compare (1.5))
\begin{equation*}
\sum_{i=1}^{n-1} f_i''(\partial_i v) \partial_i \partial_i v = 0\end{equation*}
now on $\mathbb{R}^{n-1}$. Repeating the foregoing arguments, it follows that
\begin{equation*}
v(x_1, \dots, x_{n-1}) = v(x_1, \dots, x_{n-2}, 0) + c_{n-1} x_{n-1}.
\end{equation*}
After $n-1$ steps we end up with
\begin{equation*}u(x_1, \dots, x_n) = u(x_1, 0, \dots, 0) + \sum_{i=2}^n c_i x_i
\end{equation*}
the function $\gamma(t) := u(t, 0, \dots, 0)$ being a solution $\frac{d}{dt}(f_1'(\gamma')) = 0$ on $\mathbb{R}$. This gives the constancy of $\gamma'$ 
and we finally deduce our claim
\begin{equation*}
u(x) = \sum_{i=1}^n c_i x_i + c_0,\quad x \in \mathbb{R}^n, \text{ with suitable constants } c_0, c_1, \dots, c_n \in \mathbb{R}.
\end{equation*}\qed

\
\vfill

\begin{minipage}[t]{0.48\textwidth}

\textbf{Michael Bildhauer}\\
Saarland University\\
Department of Mathematics\\
P.O. Box 15 11 50\\
66041 Saarbrücken, Germany\\
\texttt{bibi@math.uni-sb.de}

\end{minipage}
\hfill
\begin{minipage}[t]{0.48\textwidth}

\textbf{Martin Fuchs}\\
Saarland University\\
Department of Mathematics\\
P.O. Box 15 11 50\\
66041 Saarbrücken, Germany\\
\texttt{fuchs@math.uni-sb.de}

\end{minipage}
\end{document}